\documentclass[11pt]{article}

\usepackage[english]{babel}
\usepackage{amsmath,amsthm,amssymb,amsfonts}
\usepackage{bm,bbm}
\usepackage{graphicx}
\usepackage{subcaption}
\usepackage[hyperfootnotes=false]{hyperref}
\usepackage[usenames,dvipsnames,svgnames,table]{xcolor}
\usepackage[symbol]{footmisc}

\numberwithin{equation}{section}
\theoremstyle{plain}
\newtheorem{Theorem}{Theorem}[section]
\newtheorem{Lemma}[Theorem]{Lemma}

\theoremstyle{definition}

\newcommand{\E}{\mathbb{E}}

\newcommand{\tr}{\operatorname{tr}}

\newcommand{\mN}{\mathcal{N}}

\newcommand{\ba}{\mathbf{a}}

\newcommand{\bv}{\mathbf{v}}

\newcommand{\bu}{\mathbf{u}}
\newcommand{\bU}{\mathbf{U}}

\newcommand{\bx}{\mathbf{x}}
\newcommand{\bX}{\mathbf{X}}

\newcommand{\bZ}{\mathbf{Z}}
\newcommand{\bI}{\mathbf{I}}

\newcommand{\bxi}{\bm{\xi}}
\newcommand{\btheta}{\bm{\theta}}

\newcommand{\cmmnt}[1]{\ignorespaces}

\begin{document}

\title{Convergence guarantees for forward gradient descent in the linear regression model}
\author{%
    Thijs Bos\footnotemark[1] \footnotemark[3] %
    \ \ and Johannes Schmidt-Hieber\footnotemark[2] \footnotemark[3]%
}
  \footnotetext[1]{Leiden University} 
  \footnotetext[2]{University of Twente}
  \footnotetext[3]{The research has been supported by the NWO/STAR grant 613.009.034b and the NWO Vidi grant VI.Vidi.192.021.}
\date{}
\maketitle

\begin{abstract}
Renewed interest in the relationship between artificial and biological neural networks motivates the study of gradient-free methods. Considering the linear regression model with random design, we theoretically analyze in this work the biologically motivated (weight-perturbed) forward gradient scheme that is based on random linear combination of the gradient. If $d$ denotes the number of parameters and $k$ the number of samples, we prove that the mean squared error of this method converges for $k\gtrsim d^2\log(d)$ with rate $d^2\log(d)/k.$ Compared to 
the dimension dependence $d$ for stochastic gradient descent, an additional factor $d\log(d)$ occurs.
\end{abstract}




\section{Introduction}

Looking at the past developments, it is apparent that artificial neural networks (ANNs) became more powerful the more they resembled the brain. It is therefore anticipated that the future of AI is even more biologically inspired. As in the past, the bottlenecks towards more biologically inspired learning are computational barriers. For instance, shallow networks only became computationally feasible after the backpropagation algorithm was proposed. Deep neural networks were proposed for a longer time but deep learning only became scalable to large datasets after the introduction of large scale GPU computing. Neuromorphic computing aims to imitate the brain on computer chips, but is currently not fully scalable due to computational barriers.

The mathematics of AI has focused on explaining the state-of-the-art performance of modern machine learning methods and empirically observed phenomena such as the good generalization properties of extreme overparametrization. To shape the future of AI, statistical theory needs more emphasis on anticipating future developments. This includes proposing and analyzing biologically motivated methods already at a stage before scalable implementations exist. 

This work aims to analyze a biologically motivated learning rule building on the renewed interest of the differences and similarities between ANNs and biological neural networks (BNNs) \cite{BackpropAndBrain,schmidt2023interpreting,LocalHebbianPlasticity} which are rooted in the foundational literature from the 1980s \cite{GROSSBERG198723, ExcitementCrick}. A key difference between ANNs and BNNs is that ANNs are usually trained based on a version of (stochastic) gradient descent, while this seems prohibitive for BNNs. Indeed, to compute the gradient, knowledge of all parameters in the network is required, but biological networks do not posses the capacity to transport this information to each neuron. This suggests that biological networks cannot directly use the gradient to update their parameters \cite{ExcitementCrick,BackpropAndBrain, FundamentalsCompNeuro}.

The brain still performs well without gradient descent and can learn tasks with much fewer examples than ANNs. This sparks interest in biologically plausible learning methods that do not require (full) access of the gradient. Such methods are called derivative-free. A simple example of a derivative-free method is to randomly sample in each step a new parameter. If this decreases the loss, one keeps the new parameter and otherwise discards it without updating step. There is a wide variety of derivative-free strategies \cite{IntroDerivativeFreeConnea, DerivativeFreeLarsonea, IntroStochasticSearch}. Among those, so-called zeroth-order methods use evaluations of the loss function to build a noisy estimate of the gradient. This substitute is then used to replace the gradient in the gradient descent routine \cite{liu2020primer,ZeroOrderRates}. \cite{schmidt2023interpreting} establishes a connection between the Hebbian learning underlying the local learning of the brain (see e.g. Chapter 6 of \cite{FundamentalsCompNeuro}) and a specific zeroth-order method. A statistical analysis of this zeroth-order scheme is provided in the companion article \cite{SHKoolen2023}.

In this article, we study
(weight-perturbed) forward gradient descent. This method is motivated by biological neural networks \cite{baydin2022gradients, ren2022scaling} and lies between full gradient descent methods and derivative-free methods, as only random linear combination of the gradient are required. The form of the random linear combination is related to zeroth-order methods, see Section \ref{S: Model Description}. Settings with partial access to the gradient have been studied before. For example, \cite{GradientFreeMinimization} proposes a learning method based on directional derivatives for convex functions. In this work, we specifically  derive theoretical guarantees for forward gradient descent in the linear regression model with random design. Theorem \ref{T: Expectations} establishes an expression for the expectation. A bound on the mean squared error is provided in Theorem \ref{T:Convergence bound}.

The structure of the paper is as follows. In Section \ref{S: Model Description} we describe the forward gradient descent update rule in the linear regression model. Results are in Section \ref{S: Main Results} and the corresponding proofs can be found in Section \ref{S: Proofs}.

\textbf{Notation:} Vectors are denoted by bold letters and we write $\|\cdot\|_2$ for the Euclidean norm. We denote the largest and smallest eigenvalue of a matrix $A$ by the respective expressions $\lambda_{\max}(A)$ and $\lambda_{\min}(A)$. The spectral norm is $\|A\|_S:=\sqrt{\lambda_{\max}(A^\top A)}.$
The condition number of a positive semi-definite matrix $B$ is $\kappa(B):=\lambda_{\max}(B)/\lambda_{\min}(B).$ For a random variable $U,$ we denote the expectation with respect to $U$ by $\mathbb{E}_U.$ The symbol $\mathbb{E}$ 
 stands for an expectation taken with respect to all random variables that are inside that expectation. The (multivariate) normal distribution with mean vector $\mu$ and covariance matrix $\Sigma$ is denoted by $\mN(\mu,\Sigma).$

\section{Weight-perturbed forward gradient descent}\label{S: Model Description}

Suppose we want to learn a parameter vector $\btheta$ from training data $(\bX_1,Y_1),(\bX_2,Y_2),\ldots \in \mathbb{R}^{d}\times \mathbb{R}.$ Stochastic gradient descent (SGD) is based on the iterative update rule
\begin{equation}
    \btheta_{k+1}=\btheta_k-\alpha_{k+1}\nabla L(\btheta_k), \quad k=0,1,\ldots
    \label{eq.SGD}
\end{equation}
with $\btheta_0$ some initial value and $L(\btheta_k):=L(\btheta_k,\bX_k,Y_k)$ a loss that depends on the data only through the $k$-th sample $(\bX_k, Y_k)$.

For a standard normal random vector $\bxi_{k+1}\sim \mN(0,\bI_d)$ that is independent of all the other randomness, the quantity $(\nabla L(\btheta_k))^\top \bxi_{k+1} \bxi_{k+1}$ is called the (weight-perturbed) forward gradient \cite{baydin2022gradients, ren2022scaling}. {\it (Weight-perturbed) forward gradient descent} is then given by the update rule
\begin{equation}\label{Eq: Update rule}
    \btheta_{k+1}=\btheta_k-\alpha_{k+1}\big(\nabla L(\btheta_k)\big)^\top \bxi_{k+1} \bxi_{k+1}, \quad k=0,1,\ldots
\end{equation}

Assuming that the exogenous noise has unit variance is sufficient. Indeed, generalizing to $\bxi_{k+1}\sim \mN(0,\sigma^2 \bI_d)$ with variance parameter $\sigma^2$ has the same effect as rescaling the learning rate $\alpha_{k+1}\to \sigma^{-2}\alpha_{k+1}.$

Since for a deterministic $d$-dimensional vector $\bv,$ one has $\E[\bv^t\bxi_{k+1}\bxi_{k+1}]=\bv,$ taking the expectation of the weight-perturbed forward gradient descent scheme with respect to the exogenous randomness induced by $\bxi_1,\bxi_2,\ldots$ gives 
\begin{align}
    \E_{(\bxi_i)_{i\geq 1}}[\btheta_{k+1}]=\E_{(\bxi_i)_{i\geq 1}}[\btheta_k]-\alpha_{k+1} \E_{(\bxi_i)_{i\geq 1}}[\nabla L(\btheta_k)],  
    \label{eq.fGD_expectation}
\end{align}
resembling the SGD dynamic \eqref{eq.SGD}. If $\nabla L(\btheta_k)$ depends on $\btheta_k$ linearly then also $\E_{(\bxi_i)_{i\geq 1}}[\nabla L(\btheta_k)]= \nabla L(\E_{(\bxi_i)_{i\geq 1}}[\btheta_k]).$

While in expectation, forward gradient descent is related to SGD, the induced randomness of the $d$-dimensional random vectors $\bx_{k+1}$ induces a large amount of noise. To control the high noise level in the dynamic is the main obstacle in the mathematical analysis. One of the implications is that one has to make small steps by choosing a small learning rate to avoid completely erratic behavior. This particularly effects the first phase of the learning.

First order multivariate Taylor expansion shows that $L(\btheta_k+\bxi_k)-L(\btheta_k)$ and $(\nabla L(\btheta_k))^\top \bxi_{k+1}$ are close. Therefore, forward gradient descent is related to the zeroth-order method
\begin{equation}
    \btheta_{k+1}=\btheta_k-\alpha_{k+1} \big(L(\btheta_k+\bxi_k)-L(\btheta_k)\big)\bxi_k,
    \label{eq.8yegf}
\end{equation}
\cite{liu2020primer}. Consequently, forward gradient descent can be viewed as an intermediate step between gradient descent, with full access to the gradient, and zeroth-order methods that are solely based on (randomly) perturbed function evaluations.

\begin{figure}[!ht]
\begin{center}
\includegraphics[width=0.98\linewidth]{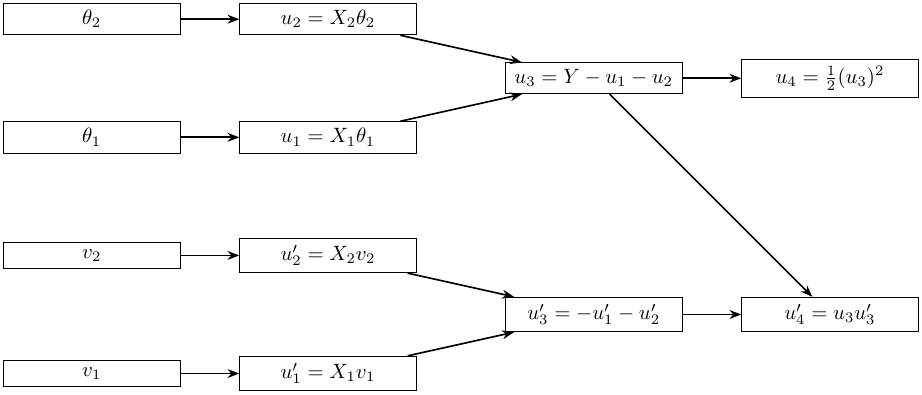}
\end{center}
\caption{Computional graphs for computing in a forward pass $L(\btheta)=\frac{1}{2}(Y-X_1\theta_1-X_2\theta_2)^2$ (upper half) and $(\nabla L(\btheta))^\top\bv$ (lower half).
}
\label{fig:ComputionalGraphs}
\end{figure}

We now comment on the biological plausibility of forward gradient descent. As mentioned in the introduction, it is widely accepted that the brain cannot perform (full) gradient descent. The backpropagation algorithm decomposes the computation of the gradient in a forward pass and a backward pass. The forward pass evaluates the loss for a training sample by sending signal through the network. This is biologically plausible. For a given vector $\bv$, it is even possible to compute both $L(\btheta_k)$ and $\big(\nabla L(\btheta_k)\big)^\top \bv$ in one forward pass, \cite{baydin2022gradients,ren2022scaling,JMLR:v18:17-468}. The construction can be conveniently explained for two variables $\btheta=(\theta_1,\theta_2)^\top,$ see Figure \ref{fig:ComputionalGraphs}. The loss function $L(\btheta)=\frac{1}{2}(Y-X_1\theta_1-X_2\theta_2)^2$ is implemented by first computing $u_1=X_1\theta_1$ and $u_2=X_2\theta_2$ in parallel. Subsequently, one can infer $u_3=Y-u_1-u_2=Y-X_1\theta_1-X_2\theta_2$ and $u_4=\frac{1}{2}(u_3)^2=L(\btheta).$ For a given vector $\bv=(v_1,v_2)^\top,$ the update value $(\nabla L(\btheta))^\top\bv$ in the forward gradient descent routine can be computed from $v_1,v_2,$ and $u_3=Y-X_1\theta_1-X_2\theta_2.$
Indeed, after computing $X_1v_1$  and $X_2v_2$ in a first step, one can compute $u_3'=-X_1v_1-X_2v_2$ and finally $u_4'=u_3u_3'=(Y-X_1\theta_1-X_2\theta_2)(-X_1v_1-X_2v_2)=-(Y-\bX^\top \btheta)\bX^\top \bv=(\nabla L(\btheta))^\top\bv.$ For more background on the implementation, see for instance \cite{JMLR:v18:17-468}.

In \cite{schmidt2023interpreting}, it has been shown that under appropriate conditions, Hebbian learning of excitatory neurons in biological neural networks leads to a zeroth-order learning rule that has the same structure as \eqref{eq.8yegf}.

To complete this section, we briefly compare forward gradient descent with feedback alignment as both methods are motivated by biological learning and are based on additional randomness.  Inspired by biological learning, feedback alignment proposes to replace the learned weights in the backward pass by random weights chosen at the start of the training procedure \cite{lillicrap2016random, BackpropAndBrain}. The so-called direct feedback alignment method goes even further: instead of back-propagating the gradient through all the layers of the network by the chain-rule, layers are updated with the gradient of the output layer multiplied with a fixed random weight matrix \cite{NIPS2016_d490d7b4, NEURIPS2020_69d1fc78}. (Direct) feedback alignment causes the forward weights to change in such a way that the true gradient of the network weights and the substitutes used in the update rule become more aligned \cite{lillicrap2016random,NIPS2016_d490d7b4,BackpropAndBrain}. The linear model can be viewed as neural network without hidden layers. The absence of layers means that in the backward step, no weight information is transported between different layers. As a consequence, both feedback alignment and direct feedback alignment collapse in the linear model into standard gradient descent. The conclusion is that feedback alignment and forward gradient descent are not comparable. The argument also shows that to unveil nontrivial statistical properties of feedback alignment, one has to go beyond the linear model. We leave the statistical analysis as an open problem.

\section{Convergence rates in the linear regression model}\label{S: Main Results}

We analyze weight-perturbed forward gradient descent for data generated from the $d$-dimensional linear regression with Gaussian random design. In this framework, we observe i.i.d.\ pairs $(\bX_i,Y_i)\in \mathbb{R}^d\times\mathbb{R},$ $i=1,2,\ldots$ satisfying
\begin{equation}\label{Eq: Regression model}
    \bX_i\sim \mN(0,\Sigma), \quad Y_i=\bX_i^\top\btheta_{\star}+\epsilon_i, \quad i=1,2,\ldots
\end{equation}
with $\btheta_{\star}$ the unknown $d$-dimensional regression vector, $\Sigma$ an unknown covariance matrix, and independent noise variables $\epsilon_i$ with mean zero and variance one.

For the analysis, we consider the squared loss $L(\btheta_k,\bX_k,Y_k)=\tfrac 12 (Y_k- \bX_k^\top\btheta_k)^2.$ The gradient is given by
\begin{align}
    \nabla L(\btheta_k) 
    = -\big(Y_k-\bX_k^\top \btheta_k \big)\bX_k.
    \label{eq.grad_loss}
\end{align}

We now analyze the forward gradient estimator assuming that the initial value $\btheta_0$ can be random or deterministic but should be independent of the data. We employ a similar proving strategy as in the recent analysis of dropout in the linear model in \cite{2023arXiv230610529C}. In particular, we will derive a recursive formula for $\mathbb{E}\left[(\btheta_k-\btheta_{\star})(\btheta_k-\btheta_{\star})^\top\right].$ In contrast to this work, we consider a different form of noise and non-constant learning rates.

The first result shows that forward gradient descent does gradient descent in expectation.

\begin{Theorem}\label{T: Expectations}
We have $\mathbb{E}[\btheta_{k}]-\btheta_{\star}=\big(\bI_d-\alpha_k\Sigma\big)\big(\mathbb{E}[\btheta_{k-1}]-\btheta_{\star}\big)$ and thus
\begin{equation}\label{Eq: Expectation single entry squared loss covariance matrix case non recursive}
    \mathbb{E}[\btheta_k]=\btheta_{\star}+\left(\prod_{\ell=1}^k(\bI_d-\alpha_{\ell}\Sigma)\right)\big(\mathbb{E}[\btheta_0]-\btheta_{\star}\big).
\end{equation}

\end{Theorem}
The proof does not exploit the Gaussian design and only requires that $\bX_{i}$ is centered and has covariance matrix $\Sigma$. The exogenous randomness induced by $\bxi_1,\bxi_2,\ldots$ disappears in the expected values but heavily influences the recursive expressions for the squared expectations.

\begin{Theorem}\label{T:Covariance}
    Consider forward gradient descent \eqref{Eq: Update rule}. If $A_k:=\mathbb{E}\left[(\btheta_k-\btheta_{\star})(\btheta_k-\btheta_{\star})^\top\right],$ then
        \begin{equation*}
    \begin{aligned}
        A_k=&(\bI_d-\alpha_{k}\Sigma)A_{k-1}(\bI_d-\alpha_k\Sigma)\\
        &+3\alpha_k^2\Sigma A_{k-1}\Sigma+2\alpha_{k}^2\mathbb{E}\big[(\btheta_{k-1}-\btheta_{\star})^\top\Sigma(\btheta_{k-1}-\btheta_{\star})\big]\Sigma+2\alpha_{k}^2\Sigma\\
        &+2\alpha_k^2\tr\big(\Sigma A_{k-1}\Sigma\big)\bI_d+\alpha_{k}^2\mathbb{E}\big[(\btheta_{k-1}-\btheta_{\star})^\top\Sigma(\btheta_{k-1}-\btheta_{\star})\big]\tr\big(\Sigma\big)\bI_d\\
        &+\alpha_k^2\tr(\Sigma)\bI_d.
    \end{aligned}
\end{equation*}
\end{Theorem}
Since $A_k$ depends on $\btheta_k^2$, the fourth moments of the design vectors $\bX_i$ and the exogenous random vectors $\bxi_k$ play a role in this equation.

The risk $\mathbb{E}\big[\|\btheta_k-\btheta_{\star}\|_2^2\big]$ is the trace of the matrix $A_k$. Setting
\begin{align*}
    \kappa(\Sigma):= \frac{\|\Sigma\|_S}{\lambda_{\min}(\Sigma)}
\end{align*}
for the condition number and building on Theorem \ref{T:Covariance}, we can establish the following risk bound for forward gradient descent.

\begin{Theorem}[Mean squared error]\label{T:Convergence bound} Consider forward gradient descent \eqref{Eq: Update rule} and assume that $\Sigma$ is positive definite. For constant $a>2, $ choosing the learning rate
\begin{equation}\label{Eq: Combined choice of alpha_k}
\begin{aligned}
    \alpha_k=\frac{a\lambda_{\min}(\Sigma)}{k\lambda^2_{\min}(\Sigma)+a\|\Sigma\|_S^2(d+2)^2},
    \quad k=1,2,\ldots,
\end{aligned}
\end{equation}
yields
\begin{equation*}
\begin{aligned}
        \mathbb{E}\big[\big\|\btheta_k-\btheta_{\star}\big\|_2^2\big]&\leq \Bigg(\frac{1+a\kappa^2(\Sigma)(d+2)^2}{k+a\kappa^2(\Sigma)(d+2)^2}\Bigg)^a\mathbb{E}\big[\big\|\btheta_{0}-\btheta_{\star}\big\|_2^2\big]
        +\frac{2ea  \kappa(\Sigma) (d+2)^2}{\lambda_{\min}(\Sigma)(k+a\kappa^2(\Sigma)(d+2)^2)}.
    \end{aligned}
\end{equation*}
\end{Theorem}
Alternatively, the upper bound of Theorem \ref{T:Convergence bound} can be written as
\begin{equation*}
\begin{aligned}
        \mathbb{E}\big[\|\btheta_k-\btheta_{\star}\|_2^2\big]&\leq \Big(1-a^{-1}\lambda_{\min}(\Sigma)(k-1)\alpha_{k}\Big)^a\mathbb{E}\big[\|\btheta_{0}-\btheta_{\star}\|_2^2\big]+2e\kappa(\Sigma)(d+2)^2\alpha_{k}.
    \end{aligned}
\end{equation*}
In the upper bound, the risk $\mathbb{E}\big[\|\btheta_{0}-\btheta_{\star}\|_2^2\big]$ of the initial estimate $\btheta_0$ appears. A realistic scenario is that the entries of $\btheta_\star$ and $\btheta_0$ are all of order one. In this case, the inequality $\|\btheta_{0}-\btheta_{\star}\|_2^2\leq d\|\btheta_{0}-\btheta_{\star}\|_{\infty}^2$ shows that the risk of the initial estimate will scale with the number of parameters $d$. Taking $a=\log(d)$ (for $d\geq 8> e^2$ such that $a>\log(e^2)=2$), Theorem \ref{T:Convergence bound} implies that 
\begin{equation*}
    \mathbb{E}\big[\|\btheta_k-\btheta_{\star}\|_2^2\big]\lesssim d\Bigg(\frac{d^2\log(d)}{k}\Bigg)^{\log(d)}\mathbb{E}\big[\|\btheta_{0}-\btheta_{\star}\|_{\infty}^2\big]+\frac{d^2\log(d)}{k}.
\end{equation*}
For $k_\star=e^2d^2\log(d),$ $d^2\log(d)/k_\star=e^{-2}$ and $d(d^2\log(d)/k_\star)^{\log(d)}=1/d$. Since $d>e^2$, this means that $d\big(d^2\log(d)/k_\star\big)^{\log(d)}<d^2\log(d)/k_\star.$ Moreover, $k^{-\log(d)}$ tends faster to zero than $k^{-1}$ as $k\rightarrow\infty$. So, for $k\geq k_\star=e^2d^2\log(d),$
\begin{equation}
    d\Bigg(\frac{d^2\log(d)}{k}\Bigg)^{\log(d)}\mathbb{E}\big[\|\btheta_{0}-\btheta_{\star}\|_{\infty}^2\big]+\frac{d^2\log(d)}{k}\leq \frac{d^2\log(d)}{k}\Big(1+\mathbb{E}\big[\|\btheta_{0}-\btheta_{\star}\|_{\infty}^2\big]\Big).
    \label{eq.fnweirugib}
\end{equation}
The rate for $k\geq e^2d^2\log(d)$ is thus $d^2\log(d)/k.$ This means that forward gradient descent has dimension dependence $d^2\log(d).$ This is by a factor $d\log(d)$ worse than the minimax rate for the linear regression problem, \cite{OptimalRatesAggregation,RandomDesignRidge,ExactminimaxLinear}. In contrast, methods that have access to the gradient can achieve optimal dimension dependence in the rate, \cite{AccelerationAveraging,LSAHowFarConstantGo}. The obtained convergence rate is in line with results showing that for convex optimization problems zeroth-order methods have a higher dimension dependence, \cite{ZeroOrderRates, liu2020primer, GradientFreeMinimization}. 

We believe that faster convergence rates are obtainable if the same datapoint is assessed several times. This means that each data point is used for several updates of the forward gradient $\btheta_{k+1}=\btheta_k-\alpha_{k+1}\big(\nabla L(\btheta_k)\big)^\top \bxi_{k+1} \bxi_{k+1},$ for instance by running multiple epochs. However, in every iteration a new random direction $\bxi_{k+1}$ is sampled. We expect that if every data point is used $m\leq d$ times, one should be able to achieve the convergence rate $d^2/(km),$ up to some logarithmic terms. If this is true and if $m$ is of the order of $d,$ one could even recover the minimax rate $d/k.$ Using the same datapoints multiple times induces additional dependence among the parameter updates. To deal with this dependence is the key challenge to establish the convergence rate $d^2/(km)$.

Assuming that the covariance matrix $\Sigma$ is positive definite is standard for linear regression with random design \cite{RandomDesignRidge, ExactminimaxLinear,GMTheoremRandomDesign}.

For $k\gtrsim d^2,$ the decrease of the learning rate $\alpha_k$ is of the order $1/k$, which is the standard choice \cite{MR1993642, AveragedGyorfi, AdaptiveAlgoStochasticApprox}. A constant learning rate is used for Ruppert-Polyak averaging in \cite{AccelerationAveraging, AveragedGyorfi}.
For least squares linear regression, it is possible to achieve (near) optimal convergence with a constant (universal) stepsize \cite{bach2013non}. Conditions under which a constant (universal) stepsize in more general settings than linear least squares works or fails are investigated in \cite{LSAHowFarConstantGo}.

\begin{figure}[!ht]
\begin{center}
\begin{subfigure}[t]{0.49\textwidth}
\includegraphics[width=0.98\linewidth]{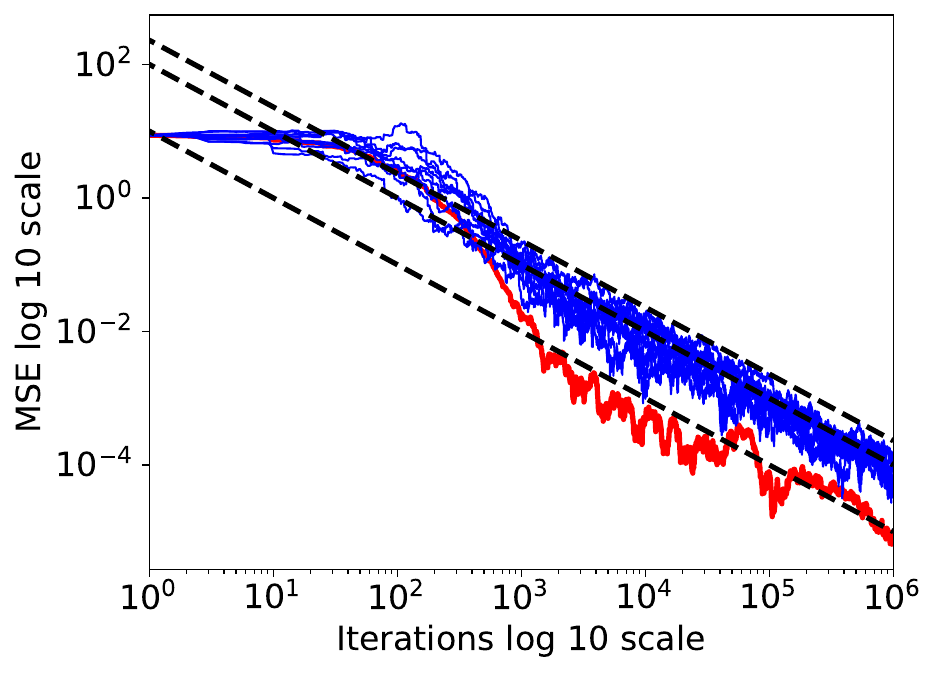}
\caption{$d=10$}
\label{fig:D10Numerical}
\end{subfigure}
\begin{subfigure}[t]{0.49\textwidth}
\includegraphics[width=0.98\linewidth]{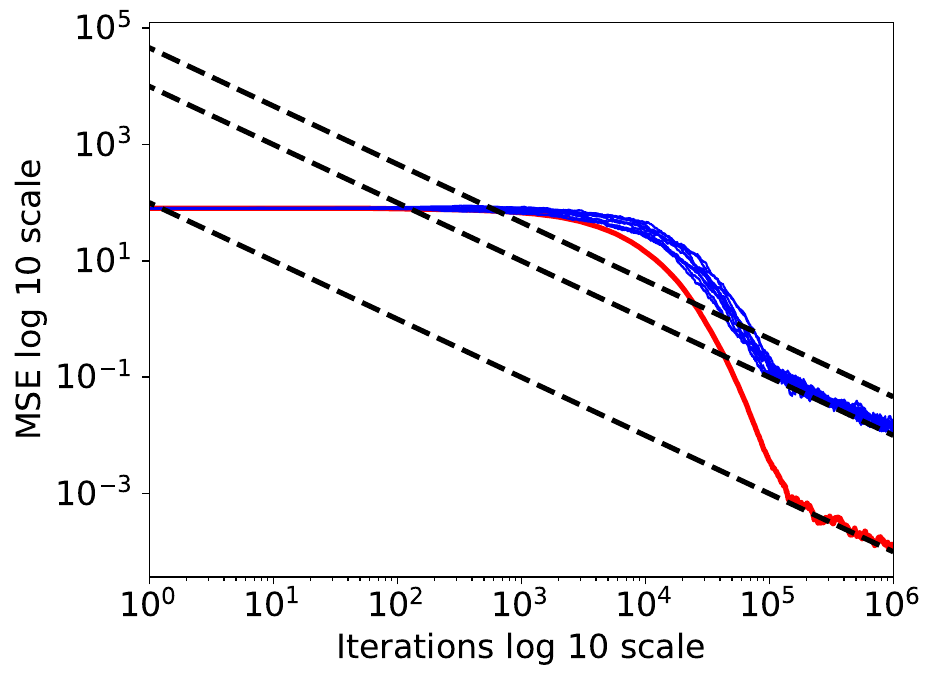}
\caption{$d=100$}
\label{fig:D100Numerical}
\end{subfigure}
\end{center}
\caption{Comparison of the MSE of forward gradient descent (blue) and SGD (red) for dimensions $d=10$ and $d=100.$ The upper dashed line is $k\mapsto d^2\log(d)/k$, the middle dashed line is $k\mapsto d^2/k$, and the lower dashed line is $k\mapsto d/k.$
}
\label{fig:NumericalResults}
\end{figure}

In a small simulation study, we investigated whether there is a discrepancy between the derived convergence rates and the empirical decay of the risk. For dimensions $d=10$ and $d=100$, data according to \eqref{Eq: Regression model} with $\Sigma=\bI_d$ are generated. On these data, we run ten times weight perturbed forward gradient descent \eqref{Eq: Update rule}, and compare the mean squared errors (MSEs) to one realization of SGD \eqref{eq.SGD}. For all simulations of forward gradient descent and SGD, we use the same initialization $\btheta_0$, drawn from a $\mN(0,\bI_d)$ distribution, and the learning rate $\alpha_k$ specified in \eqref{Eq: Combined choice of alpha_k} with $a=\log(d)$. Thus, only the random perturbation vectors $\bxi_{k}$ in the forward gradient descent schemes differ across different runs. The outcomes are reported in Figure \ref{fig:NumericalResults}. For each of the 10+1 simulations, we report on a log-log scale the MSE for the first one million iterations. The upper dashed line gives the derived convergence rate $k\mapsto d^2\log(d)/k$, the middle dashed line is $d^2/k$, and the lower dashed line is $d/k$. The ten paths from the ten forward gradient descent runs are shown in blue. The path from the SGD is displayed in red. We see three regimes. In the first regime, the risk remains nearly constant. For dimension $d=100,$ this is true up to the first ten thousand of iterations. Afterwards there is a sudden decrease of the risk. Eventually, for large number of iterations $k,$ the MSE of forward gradient descent concentrates near the line $k\mapsto d^2/k$, while the MSE of SGD concentrates around $k\mapsto d/k.$ This suggest that up to the $\log^2(d)$-factor, the derived theory does in fact describe the rate of the MSE. Equation \eqref{eq.fnweirugib} predicts that the rate $d^2\log(d)/k$ will occur for $k\geq k_\star=e^2d^2\log(d)$. For $d=10,$ $k_\star \approx 1.7\times 10^3$ and for $d=100,$ $k_\star \approx 3.4\times 10^5.$ Thus, in terms of orders of magnitude, there is a close agreement between theory and simulations. 

Starting with a good initializer that lies already in the neighborhood of the true parameter, one can avoid the long burn-in time in the beginning. Otherwise, it remains an open problem, whether one can modify the procedure such that also for smaller values of $k,$ the risk behaves more like $d^2\log(d)/k.$

Python code is available on Github \cite{Simulation_code_Convergence_2024}.

\section{Proofs}\label{S: Proofs}

\begin{proof}[Proof of Theorem \ref{T: Expectations}]
By \eqref{eq.grad_loss} and the linear regression model $Y_{k-1}=\bX_{k-1}^\top\btheta_{\star}+\epsilon_{k-1}$, we have 
\begin{align}
    \begin{split}
    \nabla L(\btheta_{k-1})
    &=-(Y_{k-1}-\bX_{k-1}^\top\btheta_{k-1})\bX_{k-1} \\
    &=-(\bX_{k-1}^\top(\btheta_\star-\btheta_{k-1})+\epsilon_{k-1})\bX_{k-1} \\
    &=-  \epsilon_{k-1} \bX_{k-1}- \bX_{k-1} \bX_{k-1}^\top(\btheta_\star-\btheta_{k-1}).    \end{split}
    \label{eq.grad_Ltheta}
\end{align}
Since $\E[\bX_{k-1} \bX_{k-1}^\top]=\Sigma,$  $\E[\epsilon_{k-1}]=0,$ and $\bX_{k-1},\epsilon_{k-1},\btheta_{k-1}$ are jointly independent, we obtain
\begin{align}
    \begin{split}
    \mathbb{E}\big[\nabla L(\btheta_{k-1})\, \big| \, \btheta_{k-1}\big]
    &= \mathbb{E}\big[ -  \epsilon_{k-1} \bX_{k-1}- \bX_{k-1} \bX_{k-1}^\top(\btheta_\star-\btheta_{k-1})\, \big| \, \btheta_{k-1}\big] \\
    &= -  \Sigma (\btheta_\star-\btheta_{k-1}).
    \end{split}
    \label{eq.grad_cond}
\end{align}
Combined with \eqref{eq.fGD_expectation}, we find
\begin{align*}
    \mathbb{E}\big[\btheta_k\big]
    &=\mathbb{E}\big[\btheta_{k-1}\big]-\alpha_k \mathbb{E}\big[\nabla L(\btheta_{k-1})\big] = \mathbb{E}\big[\btheta_{k-1}\big]+\alpha_k\Sigma \E\big[\btheta_\star-\btheta_{k-1}\big].
\end{align*}
The true parameter $\btheta_\star$ is deterministic. 
Subtracting $\btheta_\star$ on both sides, yields the claimed identity $\mathbb{E}[\btheta_{k}]-\btheta_{\star}=\big(\bI_d-\alpha_k\Sigma\big)\big(\mathbb{E}[\btheta_{k-1}]-\btheta_{\star}\big).$

\end{proof}

\subsection{Proof of Theorem \ref{T:Covariance}}
\begin{Lemma}\label{L: terms with quadratic normal random variables}
    If $\bZ\sim \mN(0,\Gamma)$ is a $d$-dimensional random vector and $\bU$ is a $d$-dimensional random vector that is independent of $\bZ,$ then
\begin{equation*}
   \mathbb{E}\big[(\bU^\top\bZ)^2\bZ\bZ^\top\big]=2\Gamma\mathbb{E}\big[\bU\bU^\top\big]\Gamma+\mathbb{E}\big[\bU^\top\Gamma\bU\big]\Gamma.
\end{equation*}
\end{Lemma}
\begin{proof}
Because $\bU$ and $\bZ$ are independent,
the $(i,j)$-th entry of the $d\times d$ matrix $\mathbb{E}\big[(\bU^\top\bZ)^2\bZ\bZ^\top\big]$ is
\begin{equation*}
    \begin{aligned}
        \sum_{\ell,m=1}^d\mathbb{E}\big[U_{\ell}U_{m}\big]\mathbb{E}\big[Z_{\ell}Z_mZ_iZ_j\big].
    \end{aligned}
\end{equation*}
Since $\bZ\sim\mN(0,\Gamma),$
\begin{equation*}
    \begin{aligned}
        \mathbb{E}\big[Z_{\ell}Z_mZ_iZ_j\big]
        &=\Gamma_{\ell,m}\Gamma_{i,j}+\Gamma_{\ell,i}\Gamma_{m,j}+\Gamma_{\ell,j}\Gamma_{m,i},
    \end{aligned}
\end{equation*}
see for instance the example at the end of Section 2 in \cite{GaussianMoments}.
Thus
\begin{equation*}
    \begin{aligned}
        &\sum_{\ell,m=1}^d\mathbb{E}\big[U_{\ell}U_{m}\big]\mathbb{E}\big[Z_{\ell}Z_mZ_iZ_j\big]=\sum_{\ell,m=1}^d\mathbb{E}\big[U_{\ell}U_{m}\big]\big(\Gamma_{\ell,m}\Gamma_{i,j}+\Gamma_{\ell,i}\Gamma_{m,j}+\Gamma_{\ell,j}\Gamma_{m,i}\big).
    \end{aligned}
\end{equation*}
Because of
\begin{align*}
    \sum_{\ell,m=1}^d\mathbb{E}\big[U_{\ell}U_{m}\big]\Gamma_{\ell,m}\Gamma_{i,j}
    &=\sum_{\ell,m=1}^d\mathbb{E}\big[U_{\ell}\Gamma_{\ell,m}U_{m}\big]\Gamma_{i,j}=\mathbb{E}\left[\bU^\top\Gamma\bU\Gamma_{i,j}\right],\\
    \sum_{\ell,m=1}^d\mathbb{E}\big[U_{\ell}U_{m}\big]\Gamma_{\ell,i}\Gamma_{m,j}
    &=\sum_{\ell,m=1}^d\mathbb{E}\big[U_{\ell}\Gamma_{\ell,i}U_{m}\Gamma_{m,j}\big]=\mathbb{E}\left[\big(\bU^\top\Gamma\big)_i\big(\bU^\top\Gamma\big)_j\right],
\end{align*}
and
\begin{equation*}
    \sum_{\ell,m=1}^d\mathbb{E}\big[U_{\ell}U_{m}\big]\Gamma_{\ell,j}\Gamma_{m,i}= \sum_{\ell,m=1}^d\mathbb{E}\big[U_{m}\Gamma_{m,i}U_{\ell}\Gamma_{\ell,j}\big]=\mathbb{E}\left[\big(\bU^\top\Gamma\big)_i\big(\bU^\top\Gamma\big)_j\right],
\end{equation*}
the $(i,j)$-th entry of the matrix $\mathbb{E}\left[(\bU^\top\bZ)^2\bZ\bZ^\top\right]$ is
\begin{equation*}
    2\mathbb{E}\left[\big(\bU^\top\Gamma\big)_i\big(\bU^\top\Gamma\big)_j\right]+\mathbb{E}\left[\bU^\top\Gamma\bU\Gamma_{i,j}\right].
\end{equation*}
For a vector $\ba=(a_1,\ldots,a_d)^\top,$ the scalar $a_ia_j$ is the $(i,j)$-th entry of the matrix $\ba\ba^\top.$ Combined with the previous display, the result follows.
\end{proof}

\begin{proof}[Proof of Theorem \ref{T:Covariance}]
As Theorem \ref{T:Covariance} only involves one update step, we can simplify the notation by dropping the index $k$ and analyzing $\btheta''=\btheta'-\alpha \big(\nabla L(\btheta')\big)^\top \bxi \bxi$ for one data point $(\bX,Y)$ and independent $\bxi\sim \mN(0,I_d).$ With $A':=\E\big[(\btheta'-\btheta_\star)(\btheta'-\btheta_\star)^\top\big]$ and $A'':=\E\big[(\btheta''-\btheta_\star)(\btheta''-\btheta_\star)^\top\big]$, we then have to prove that 
\begin{equation*}
    \begin{aligned}
        A''=&(\bI_d-\alpha\Sigma)A'(\bI_d-\alpha\Sigma)+3\alpha^2\Sigma A'\Sigma+2\alpha^2\mathbb{E}\big[(\btheta'-\btheta_{\star})^\top\Sigma(\btheta'-\btheta_{\star})\big]\Sigma+2\alpha^2\Sigma\\
        &+2\alpha^2\tr\big(\Sigma A'\Sigma\big)\bI_d+\alpha^2\mathbb{E}\big[(\btheta'-\btheta_{\star})^\top\Sigma(\btheta'-\btheta_{\star})\big]\tr\big(\Sigma\big)\bI_d+\alpha^2\tr(\Sigma)\bI_d.
    \end{aligned}
\end{equation*}
Substituting the update rule \eqref{Eq: Update rule} in $A_k$ gives by the linearity of the transpose that
\begin{equation}\label{Eq: Equation for A_k}
\begin{aligned}
    A''&=\mathbb{E}\big[(\btheta''-\btheta_{\star})(\btheta''-\btheta_{\star})^\top\big]\\
    &=\mathbb{E}\bigg[\Big(\btheta'-\alpha\big(\nabla L(\btheta')\big)^\top \bxi\bxi-\btheta_{\star}\Big)\Big(\btheta'-\alpha\big(\nabla L(\btheta')\big)^\top \bxi\bxi-\btheta_{\star}\Big)^\top\bigg]\\
    &=A' -\alpha\mathbb{E}\bigg[\Big(\btheta-\btheta_{\star}\Big)\Big(\big(\nabla L(\btheta')\big)^\top \bxi \bxi\Big)^\top\bigg]-\alpha\mathbb{E}\bigg[\Big(\big(\nabla L(\btheta')\big)^\top \bxi \bxi\Big)\Big(\btheta'-\btheta_{\star}\Big)^\top\bigg]\\
    &\quad+\mathbb{E}\bigg[\Big(\alpha\big(\nabla L(\btheta')\big)^\top \bxi \bxi\Big)\Big( \alpha\big(\nabla L(\btheta')\big)^\top \bxi \bxi\Big)^\top\bigg].
\end{aligned}
\end{equation}
First, consider the terms with the minus sign in the above expression. The random vector $\bxi$ is independent of all other randomness and hence $\mathbb{E}_{\bxi}\big[\big(\nabla L(\btheta')\big)^\top \bxi\bxi\big]=\nabla L(\btheta').$ Moreover, together with \eqref{eq.grad_cond},
\begin{align*}
    \mathbb{E}\bigg[\Big(\big(\nabla L(\btheta')\big)^\top \bxi \bxi\Big)\Big(\btheta'-\btheta_{\star}\Big)^\top \, \bigg|\, \btheta'\bigg]
    = \mathbb{E}\big[\nabla L(\btheta') \, \big|\, \btheta'\big] (\btheta'-\btheta_{\star})^\top
    =   \Sigma (\btheta'-\btheta_\star)(\btheta'-\btheta_{\star})^\top.
\end{align*}
Taking the transpose and tower rule, we find
\begin{align}\label{Eq: Bound on (I)}
    \begin{split}
    &-\alpha\mathbb{E}\bigg[\Big(\btheta-\btheta_{\star}\Big)\Big(\big(\nabla L(\btheta')\big)^\top \bxi \bxi\Big)^\top\bigg]-\alpha\mathbb{E}\bigg[\Big(\big(\nabla L(\btheta')\big)^\top \bxi \bxi\Big)\Big(\btheta'-\btheta_{\star}\Big)^\top\bigg]\\
        &=-\alpha\mathbb{E}\big[(\btheta'-\btheta_{\star})(\btheta'-\btheta_{\star})^\top\big]\Sigma-\alpha\Sigma\mathbb{E}\big[(\btheta'-\btheta_{\star})(\btheta'-\btheta_{\star})^\top\big].
    \end{split}    
\end{align}

In a next step, we derive an expression for $\mathbb{E}\Big[\Big(\alpha\big(\nabla L(\btheta')\big)^\top \bxi \bxi\Big)\Big( \alpha\big(\nabla L(\btheta')\big)^\top \bxi \bxi\Big)^\top\Big]$. Since $\bxi\sim\mN(0,\bI_d)$ is independent of $\nabla L(\btheta')$ we can apply Lemma \ref{L: terms with quadratic normal random variables} to derive
\begin{equation}\label{Eq: gradient product term expectation}
    \begin{aligned}
        &\mathbb{E}\bigg[\Big(\alpha\big(\nabla L(\btheta')\big)^\top \bxi \bxi\Big)\Big( \alpha\big(\nabla L(\btheta')\big)^\top \bxi \bxi\Big)^\top\bigg]\\
        &=\alpha^2\mathbb{E}\bigg[\Big(\big(\nabla L(\btheta')\big)^\top \bxi \Big)^2\bxi\bxi^\top\bigg]\\
        &=2\alpha^2\mathbb{E}\Big[\big(\nabla L(\btheta')\big)\big(\nabla L(\btheta')\big)^\top\Big]+\alpha^2\mathbb{E}\Big[\big(\nabla L(\btheta')\big)^\top\big(\nabla L(\btheta')\big)\Big]\bI_d \\
        &= 2\alpha^2\mathbb{E}\Big[\big(\nabla L(\btheta')\big)\big(\nabla L(\btheta')\big)^\top\Big]+\alpha^2\tr\bigg(\mathbb{E}\Big[\big(\nabla L(\btheta')\big)\big(\nabla L(\btheta')\big)^\top\Big]\bigg)\bI_d.
    \end{aligned}
\end{equation}
 Arguing as for \eqref{eq.grad_Ltheta} gives $\nabla L(\btheta')
=-  \epsilon \bX- \bX \bX^\top(\btheta_\star-\btheta')$ and this yields
\begin{equation*}
    \begin{aligned}
       \mathbb{E}\Big[\big(\nabla L(\btheta')\big)\big(\nabla L(\btheta')\big)^\top\Big]= \mathbb{E}\bigg[\mathbb{E}_{\epsilon}\Big[\big( \epsilon \bX+ \bX \bX^\top(\btheta_\star-\btheta')\big)\big( \epsilon \bX+ \bX \bX^\top(\btheta_\star-\btheta')\big)^\top\Big]\bigg].
    \end{aligned}
\end{equation*}
Because $\epsilon$ has mean zero and variance one and is independent of $(\bX,\btheta')$, we conclude that
\begin{equation}\label{Eq: gradient product term post epsilon}
    \begin{aligned}
     \mathbb{E}\Big[\big(\nabla L(\btheta')\big)\big(\nabla L(\btheta')\big)^\top\Big]&=\mathbb{E}\Big[\big(\bX \bX^\top(\btheta_\star-\btheta')\big)\big(\bX \bX^\top(\btheta_\star-\btheta')\big)^\top+\bX\bX^\top\Big]\\
       &=\mathbb{E}\Big[ \big(\bX^\top(\btheta_\star-\btheta')\big)^2\bX\bX^\top \Big]+\Sigma,
    \end{aligned}
\end{equation}
where for the last equality we used that $\bX^\top(\btheta_\star-\btheta')$ is a scalar and that $\bX\sim\mN(0,\Sigma)$. 
Since $\bX\sim\mN(0,\Sigma)$ is independent of $\btheta'$ we get by Lemma \ref{L: terms with quadratic normal random variables} that
\begin{equation*}
    \begin{aligned}
        &\mathbb{E}\Big[ \big(\bX^\top(\btheta_\star-\btheta')\big)^2\bX\bX^\top \Big]=2\Sigma\mathbb{E}\big[(\btheta'-\btheta_{\star})(\btheta'-\btheta_{\star})^\top\big]\Sigma+\mathbb{E}\big[(\btheta'-\btheta_{\star})^\top\Sigma(\btheta'-\btheta_{\star})\big]\Sigma.
    \end{aligned}
\end{equation*}
Substituting this in \eqref{Eq: gradient product term post epsilon} and \eqref{Eq: gradient product term expectation} yields
\begin{equation}\label{Eq: Equation for (II)}
    \begin{aligned}
    &\mathbb{E}\bigg[\Big(\alpha\big(\nabla L(\btheta')\big)^\top \bxi \bxi\Big)\Big( \alpha\big(\nabla L(\btheta')\big)^\top \bxi \bxi\Big)^\top\bigg]\\
        &=4\alpha^2\Sigma\mathbb{E}\big[(\btheta'-\btheta_{\star})(\btheta'-\btheta_{\star})^\top\big]\Sigma+2\alpha^2\mathbb{E}\big[(\btheta'-\btheta_{\star})^\top\Sigma(\btheta'-\btheta_{\star})\big]\Sigma+2\alpha^2\Sigma\\
        &+2\alpha^2\tr\Big(\Sigma\mathbb{E}\big[(\btheta'-\btheta_{\star})(\btheta'-\btheta_{\star})^\top\big]\Sigma\Big)\bI_d+\alpha^2\tr\Big(\mathbb{E}\big[(\btheta'-\btheta_{\star})^\top\Sigma(\btheta'-\btheta_{\star})\big]\Sigma\Big)\bI_d\\
        &+\alpha^2\tr(\Sigma)\bI_d.
    \end{aligned}
\end{equation}
Combining \eqref{Eq: Equation for A_k} with \eqref{Eq: Bound on (I)} and \eqref{Eq: Equation for (II)} yields the statement of the theorem.
\end{proof}

\subsection{Proof of Theorem \ref{T:Convergence bound}}

For two vectors $\bu,\bv$ of the same length, $\tr(\bu\bv^\top)=\bu^\top\bv.$ Thus,
$\mathbb{E}\big[\|\btheta_{k}-\btheta_{\star}\|_2^2\big]=\tr\big(\mathbb{E}\big[(\btheta_k-\btheta_{\star})(\btheta_k-\btheta_{\star})^\top\big]\big).$
Together with Theorem \ref{T:Covariance}, $\tr(\bI_d)=d$ and $\tr(AB)=\tr(BA)$ for square matrices $A$ and $B$ of the same size, this yields
\begin{equation}\label{Eq: Risk as trace of Theorem 3.2}
    \begin{aligned}
       \mathbb{E}\big[\|\btheta_{k}-\btheta_{\star}\|_2^2\big]
       &=\tr\Big((\bI_d-\alpha_{k}\Sigma)\mathbb{E}\big[(\btheta_{k-1}-\btheta_{\star})(\btheta_{k-1}-\btheta_{\star})^\top\big](\bI_d-\alpha_k\Sigma)\Big)\\
        &\quad +3\alpha_k^2\tr\Big(\Sigma \mathbb{E}\big[(\btheta_{k-1}-\btheta_{\star})(\btheta_{k-1}-\btheta_{\star})^\top\big]\Sigma\Big)\\&\quad +2\alpha_{k}^2\tr\Big(\mathbb{E}\big[(\btheta_{k-1}-\btheta_{\star})^\top\Sigma(\btheta_{k-1}-\btheta_{\star})\big]\Sigma\Big)+2\alpha_{k}^2\tr\big(\Sigma\big)\\
        &\quad +2\alpha_k^2\tr\Big(\Sigma \mathbb{E}\big[(\btheta_{k-1}-\btheta_{\star})(\btheta_{k-1}-\btheta_{\star})^\top\big]\Sigma\Big)\tr\big(\bI_d\big)\\ 
        &\quad +\alpha_{k}^2\mathbb{E}\big[(\btheta_{k-1}-\btheta_{\star})^\top\Sigma(\btheta_{k-1}-\btheta_{\star})\big]\tr\big(\Sigma\big)\tr\big(\bI_d\big)\\
        &\quad +\alpha_k^2\tr(\Sigma)\tr\big(\bI_d\big)\\
        &=\mathbb{E}\big[(\btheta_{k-1}-\btheta_{\star})^\top(\bI_d-2\alpha_{k}\Sigma)^\top(\btheta_{k-1}-\btheta_{\star})\big]\\
        &\quad +2(d+2)\alpha_k^2\tr\Big(\Sigma \mathbb{E}\big[(\btheta_{k-1}-\btheta_{\star})(\btheta_{k-1}-\btheta_{\star})^\top\big]\Sigma\Big)\\&\quad +(d+2)\alpha_k^2\Big(\mathbb{E}\big[(\btheta_{k-1}-\btheta_{\star})^\top\Sigma(\btheta_{k-1}-\btheta_{\star})\big]\tr\big(\Sigma\big)+\tr\big(\Sigma\big)\Big).
    \end{aligned}
\end{equation}
If $\lambda$ is an eigenvalue of $\Sigma,$ then $(1-2\alpha_{k}\lambda)$ is an eigenvalue of $\bI_d-2\alpha_{k}\Sigma$.  By assumption, $0<\alpha_{k}\leq \lambda_{\min}(\Sigma)/\big(2\|\Sigma\|_S^2\big)\leq 1/\big(2\lambda_{\max}(\Sigma)\big)$ and therefore the matrix $\bI_d-2\alpha_{k}\Sigma$ is positive semi-definite and $(1-2\alpha_{k}\lambda_{\min}(\Sigma))$ is the largest eigenvalue.

For a positive semi-definite matrix $A$ and a vector $\bv,$ the min-max theorem states that $\bv^\top A\bv\leq \lambda_{\max}(A)\|\bv\|_2^2=\|A\|_S\|\bv\|_2^2.$ Using that for a vector $\bx$ it holds that $\tr(\bx\bx^\top)=\bx^\top\bx$, with $\bx=\Sigma(\btheta_{k-1}-\btheta_{\star})$ in \eqref{Eq: Risk as trace of Theorem 3.2} and applying $\bv^\top A\bv\leq\|A\|_S\|\bv\|_2^2$ with $\bv=\btheta_{k-1}-\btheta_{\star}$ and $A\in \{\Sigma,\bI_d-2\alpha_k\Sigma,\Sigma^2\}$, yields
\begin{equation*}
    \begin{aligned}
        \mathbb{E}\big[\|\btheta_{k}-\btheta_{\star}\|_2^2\big]
        &\leq \big(1-2\alpha_{k}\lambda_{\min}(\Sigma)\big) \mathbb{E}\big[\|\btheta_{k-1}-\btheta_{\star}\|_2^2\big]\\
        &+(d+2)\alpha_k^2\bigg(\tr(\Sigma)\|\Sigma\|_S\mathbb{E}\big[\|\btheta_{k-1}-\btheta_{\star}\|_2^2\big]+2\|\Sigma\|_S^2\mathbb{E}\big[\|\btheta_{k-1}-\btheta_{\star}\|_2^2\big]+\tr(\Sigma)\bigg).
    \end{aligned}
\end{equation*}
The spectral norm of a positive semi-definite matrix is equal to the largest eigenvalue and so $\tr(\Sigma)=\sum_{i=1}^d\lambda_i\leq d\lambda_{\max}=d\|\Sigma\|_S.$ Therefore,
\begin{equation*}
    \begin{aligned}
        &\mathbb{E}\big[\|\btheta_{k}-\btheta_{\star}\|_2^2\big]\leq\Big(1-2\alpha_k\lambda_{\min}(\Sigma)+\|\Sigma\|_S^2(d+2)^2\alpha_k^2\Big)\mathbb{E}\big[\|\btheta_{k-1}-\btheta_{\star}\|_2^2\big]+\|\Sigma\|_S(d+2)^2\alpha_k^2.
    \end{aligned}
\end{equation*}
 Using that $\alpha_k\leq \lambda_{\min}(\Sigma)/\big(\|\Sigma\|_S^2(d+2)^2\big)$ yields
\begin{equation*}
    \begin{aligned}
        &\mathbb{E}\big[\|\btheta_{k}-\btheta_{\star}\|_2^2\big]\leq \big(1-\alpha_k\lambda_{\min}(\Sigma)\big)\mathbb{E}\big[\|\btheta_{k-1}-\btheta_{\star}\|_2^2\big]+\|\Sigma\|_S(d+2)^2\alpha_k^2.
    \end{aligned}
\end{equation*}
Rewritten in non-recursive form, we obtain
\begin{equation}\label{Eq: Nonrecursive Bound on MSE}
    \begin{aligned}
        \mathbb{E}\big[\|\btheta_{k}-\btheta_{\star}\|_2^2\big]\leq &\mathbb{E}\big[\|\btheta_{0}-\btheta_{\star}\|_2^2\big]\prod_{\ell=1}^k \big(1-\alpha_{\ell}\lambda_{\min}(\Sigma)\big)\\
        &+\|\Sigma\|_S(d+2)^2\sum_{m=0}^{k-1}\alpha_{k-m}^2\prod_{\ell=k-m+1}^{k}\big(1-\alpha_{\ell}\lambda_{\min}(\Sigma)\big),
    \end{aligned}
\end{equation}
where we use the convention that the (empty) product over zero terms is given the value $1.$
For ease of notation define $c_d:=a\kappa^2(\Sigma)(d+2)^2,$ with condition number $\kappa(\Sigma)=\|\Sigma\|_S/\lambda_{\min}(\Sigma).$ From the definition of $\alpha_k$, \eqref{Eq: Combined choice of alpha_k}, it follows that $\alpha_k=\frac{a}{\lambda_{\min}(\Sigma)}\cdot \frac{1}{k+c_d}.$ Using that for all real numbers $x$ it holds that $1+x\leq e^x$, we get that for all integers $k^*< k,$
\begin{equation}\label{Eq: Bound on the product term alternative}
    \begin{aligned}
        &\prod_{\ell=k^*}^{k}\big(1-\alpha_{\ell}\lambda_{\min}(\Sigma)\big)\leq \exp\Bigg(-\lambda_{\min}(\Sigma)\sum_{\ell=k^*}^k\alpha_{\ell}\Bigg)=\exp\Bigg(-a\sum_{\ell=k^*}^k\frac{1}{\ell+c_d}\Bigg).
    \end{aligned}
\end{equation}
The function $x\mapsto 1/(x+c)$ is monotone decreasing for $x>0$ and $c\geq 0$ and thus,
\begin{equation}\label{Eq: Lower bound on the sum in the exponential alternative}
\begin{aligned}
    \sum_{\ell=k^*}^k\frac{1}{\ell+c_d}
    &\geq \sum_{\ell=k^*}^k\int_{\ell}^{\ell+1}\frac{1}{x+c_d}dx \\
    &= \int_{k^*}^{k+1}\frac{1}{x+c_d}dx\\
    &=\log(k+1+c_d)-\log(k^*+c_d)\\
    &=\log\Big(\frac{k+1+c_d}{k^*+c_d}\Big). 
\end{aligned}
\end{equation}
Using \eqref{Eq: Bound on the product term alternative} and \eqref{Eq: Lower bound on the sum in the exponential alternative} with $k^*=1$ gives
\begin{equation}\label{Eq: Bound on the product for the MSE(btheta0) term Alternative}
    \begin{aligned}
    \prod_{\ell=1}^k \big(1-\alpha_{\ell}\lambda_{\min}(\Sigma)\big)\leq \exp\bigg(-a\log\Big(\frac{k+1+c_d}{1+c_d}\Big)\bigg)=\bigg(\frac{1+c_d}{k+1+c_d}\bigg)^a.
    \end{aligned}
\end{equation}

Using \eqref{Eq: Bound on the product term alternative} and \eqref{Eq: Lower bound on the sum in the exponential alternative} with $k^*=k-m+1$ gives
\begin{equation}\label{Eq: Bound on the sum product term alternative}
    \begin{aligned}
        &\sum_{m=0}^{k-1}\alpha_{k-m}^2\prod_{\ell=k-m+1}^{k}\big(1-\alpha_{\ell}\lambda_{\min}(\Sigma)\big)\\
        &\leq \frac{a^2}{\lambda^2_{\min}(\Sigma)}\sum_{m=0}^{k-1}\frac{1}{\big((k-m)+c_d\big)^2}\Bigg(\frac{k-m+1+c_d}{k+1+c_d}\Bigg)^a\\
        &=\frac{a^2}{\lambda^2_{\min}(\Sigma)(k+1+c_d)^a}\sum_{m=0}^{k-1}\frac{\big(k-m+1+c_d\big)^a}{\big((k-m)+c_d\big)^2}\\
        &=\frac{a^2}{\lambda^2_{\min}(\Sigma)(k+1+c_d)^a}\sum_{m=1}^{k}\frac{\big(m+1+c_d\big)^a}{\big(m+c_d\big)^2}.
    \end{aligned}
\end{equation}

Observe that $c_d=a\kappa^2(\Sigma)(d+2)^2\geq a$. This gives us that $c_d+1\leq (1+1/a)c_d$ and thus $m+1+c_d\leq (1+1/a)(m+c_d)$. For all real numbers $x,$ $(1+x)\leq e^x$ and thus $(1+1/a)^a\leq e.$ Therefore,
\begin{equation}\label{Eq: Step in alternative bound extra term}
    \begin{aligned}
       &\sum_{m=1}^{k}\frac{\big(m+1+c_d\big)^a}{\big(m+c_d\big)^2}\leq e\sum_{m=1}^{k}\big(m+c_d\big)^{a-2}.    \end{aligned}
\end{equation}

For $p>0,$ the function $x\mapsto (x+c)^p$ is monotone increasing for $x,c>0,$ Hence,
    \begin{equation*}
    \begin{aligned}
      \sum_{\ell=1}^k(\ell+c)^{p}
      &\leq \sum_{\ell=1}^{k}\int_{\ell}^{\ell+1}(x+c)^{p}dx\\
      &=\int_1^{k+1}(x+c)^{p}dx\\
      &=\frac{(k+1+c)^{p+1}}{p+1}-\frac{(1+c)^{p+1}}{p+1} \\
      &\leq \frac{(k+1+c)^{p+1}}{p+1}.
    \end{aligned}
    \end{equation*}
Since $a>2,$ we can apply this with $p=a-2>0,$ to find
\begin{equation*}
    e\sum_{m=1}^{k}\big(m+c_d\big)^{a-2}\leq e\frac{(k+1+c_d)^{a-1}}{a-1}.
\end{equation*}
Combining \eqref{Eq: Nonrecursive Bound on MSE}, \eqref{Eq: Bound on the product for the MSE(btheta0) term Alternative}, \eqref{Eq: Bound on the sum product term alternative} and \eqref{Eq: Step in alternative bound extra term} finally gives
\begin{equation*}
\begin{aligned}
        \mathbb{E}[\|\btheta_k-\btheta_{\star}\|_2^2]&\leq \Bigg(\frac{1+a\kappa^2(\Sigma)(d+2)^2}{k+1+a\kappa^2(\Sigma)(d+2)^2}\Bigg)^a\mathbb{E}[\|\btheta_{0}-\btheta_{\star}\|_2^2]\\
        &\quad +\frac{ea^2\kappa(\Sigma)(d+2)^2}{\lambda_{\min}(\Sigma)\big(a-1\big)\big(k+1+a\kappa^2(\Sigma)(d+2)^2\big)}.
    \end{aligned}
\end{equation*}
Using that $0<a/(a-1)<2$ for $a>2,$ now yields the result.
\qed

\bibliographystyle{acm}
\bibliography{AlternativeOptimationforRegressionLibrary}

\end{document}